\documentclass[11pt]{amsart}
\usepackage{amssymb,amsmath,amsthm,newlfont,enumerate}

\theoremstyle{plain}
\newtheorem{theorem}{Theorem}[section]
\newtheorem{lemma}[theorem]{Lemma}
\newtheorem{corollary}[theorem]{Corollary}

\theoremstyle{remark}

\newtheorem*{remarks}{Remarks}

\usepackage[colorlinks = true]{hyperref}

\newcommand{\TT}{{\mathbb T}}
\newcommand{\ZZ}{{\mathbb Z}}
\newcommand{\DD}{{\mathbb D}}
\newcommand{\cD}{{\mathcal D}}
\newcommand{\cH}{{\mathcal H}}
\DeclareMathOperator{\spn}{span}
\DeclareMathOperator{\supp}{supp}
\DeclareMathOperator{\hol}{Hol}

\begin{document}

\date{22 June 2020}
\title{Failure of approximation of odd functions by odd polynomials}

\author{Javad Mashreghi}
\address{D\'epartement de math\'ematiques et de statistique, Universit\'e Laval,
Qu\'ebec City (Qu\'ebec),  Canada G1V 0A6.}
\email{javad.mashreghi@mat.ulaval.ca}

\author{Pierre-Olivier Paris\'e}
\address{D\'epartement de math\'ematiques et de statistique, Universit\'e Laval,
Qu\'ebec City (Qu\'ebec),  Canada G1V 0A6.}
\email{pierre-olivier.parise.1@ulaval.ca}

\author{Thomas Ransford}
\address{D\'epartement de math\'ematiques et de statistique, Universit\'e Laval,
Qu\'ebec City (Qu\'ebec),  Canada G1V 0A6.}
\email{thomas.ransford@mat.ulaval.ca}

\thanks{JM supported by an NSERC Discovery Grant. 
POP supported by an NSERC Alexander-Graham-Bell Scholarship.
TR supported by grants from NSERC and the Canada Research Chairs program.}

\begin{abstract}
We construct a Hilbert holomorphic function space $H$ on the unit disk such that the polynomials
are dense in $H$, but the odd polynomials are not dense in the odd functions in $H$.
As a consequence, there exists a function $f$ in $H$ that lies outside the closed linear span of its
Taylor partial sums $s_n(f)$, so it cannot be approximated by any
triangular summability method applied to the $s_n(f)$.
We also show that there exists a function $f$ in $H$
that lies outside the closed linear span of its radial dilates $f_r, ~r<1$.
\end{abstract}



\maketitle

\section{Introduction and statement of main results}\label{S:intro}

We denote by $\DD$ the open unit disk, and by $\hol(\DD)$ 
the Fr\'echet space of holomorphic functions on $\DD$,
equipped with the topology of uniform convergence on compact sets. 
A \emph{Hilbert holomorphic function space on $\DD$} 
is a Hilbert space $H$ that is a subset of $\hol(\DD)$,
such that the inclusion map $H\hookrightarrow\hol(\DD)$ is continuous.

There are many examples of such spaces,
including the Hardy space \cite{CR00,Ni19}, the Dirichlet space \cite{EKMR14}, 
the Bergman space \cite{HKZ00},
the local Dirichlet spaces \cite{RS91},
as well as the de Branges--Rovnyak spaces \cite{Sa94}.
Most of these spaces contain the polynomials
as a dense subspace.

Our main result is the following surprising theorem.
As usual, we say that a function $f$  is \emph{odd} if $f(-z)=-f(z)$.

\begin{theorem}\label{T:horrible}
There exists a Hilbert holomorphic function space $H$  on $\DD$ 
such that:
\begin{itemize}
\item $H$ contains the polynomials,
\item the polynomials are dense in $H$,
\item the odd polynomials are \emph{not} dense in the odd functions in $H$. 
\end{itemize}
Moreover, given any positive sequence $(\omega_n)_{n\ge0}$
such that $\sum_n 1/\omega_{n}<\infty$,
the space $H$ can be chosen so that $\|z^n\|_H\le 1+\omega_n$ for all $n\ge0$.
\end{theorem}

We shall prove this result in \S\ref{S:proof}, and explore some variants in \S\ref{S:variations}.
For example can we replace the odd functions by the even functions?
Can we apply these ideas to other Hilbert spaces, for instance spaces arising in Fourier analysis?
In the rest of this section, we explore the implications of Theorem~\ref{T:horrible}.

In most spaces, density of polynomials is proved by a direct construction.
For example, in the Hardy space $H^2$,  every function $f$ 
can be approximated by the partial sums  $s_n(f)$ of its Taylor expansion. 
Clearly the $s_n(f)$ are polynomials,
and,  from the very definition of the $H^2$-norm, 
we have that $\|s_n(f)-f\|_{H^2}\to0$ as $n\to\infty$.
A similar argument works in the classical Bergman and Dirichlet spaces.

However, there are spaces $H$ where this simple procedure breaks down, 
even though the polynomials are dense. 
One such space is $H=\cD_\zeta$,
the local Dirichlet space at a point~$\zeta$ in the unit circle.
As  in the well-known case of the disk algebra,
there exists a function $f\in\cD_\zeta$ for which $\sup_n\|s_n(f)\|_{\cD_\zeta}=\infty$,
so in particular $s_n(f)$ does not converge to $f$ in norm.
Just as for the disk algebra, however, the Ces\`aro means
\[
\sigma_n(f):=\frac{s_0(f)+s_1(f)+\cdots+s_n(f)}{n+1}
\]
do always converge to $f$ in norm.
For details, we refer to \cite{MR19b}.

Even worse is the case when $H=\cH(b)$, a de Branges--Rovnyak space.
Then it can happen that, even though polynomials are dense, neither the Taylor partial sums $s_n(f)$,
nor their Ces\`aro means $\sigma_n(f)$ converge to $f$.
Worse still,  the radial dilates $f_r(z):=f(rz)$ may tend to infinity in norm as $r\to1^-$,
even though they clearly converge pointwise to $f$. For details, we refer to \cite{EFKMR16}.

There is however a general approximation procedure that works in any Hilbert holomorphic function
space $H$ in which the polynomials are dense.
Using the Gram--Schmidt process, we can find an orthonormal basis $(p_n)$ of $H$ consisting of polynomials.
If we define $T_n:H\to H$ by 
\[
T_n(f):=\sum_{k=1}^n \langle f,p_k\rangle p_k,
\]
then $(T_n)$ is a \emph{linear polynomial approximation scheme} for $H$,
namely a sequence of bounded linear self-maps of $H$ with the property that,
 for every $f\in H$, the sequence $(T_n(f))$ consists of polynomials converging to $f$ in the norm of $H$. 
It was shown in \cite{MR19a} that,
more generally, every Banach holomorphic function space $X$ on $\DD$
admits a linear polynomial approximation scheme $(T_n)$, provided that merely polynomials are dense  
and that $X$ has the bounded approximation property.
It was further shown that $T_n$ can always be chosen so that $\deg T_n(f)\le n$ for all $n$.
In view of this, it seems reasonable to ask if $T_n$ can be chosen to have the form
\begin{equation}\label{E:triangular}
T_n(f)=\sum_{k=0}^n c_{nk}s_k(f)
\end{equation}
for some triangular array of complex numbers $(c_{nk})_{0\le k\le n<\infty}$.
The following result answers this question in the negative, 
even in the special case of a Hilbert holomorphic function space.

\begin{corollary}\label{C:summability}
Let $H$ be as in Theorem~\ref{T:horrible}.
Then, despite the fact that polynomials are dense in $H$,
there exists $f\in H$ lying outside the closed linear span of $\{s_n(f):n\ge0\}$.
Hence there is no sequence of linear maps $T_n :H\to H$ of the form \eqref{E:triangular}
such that $\|T_n(f)-f\|_H\to0$.
\end{corollary}

\begin{proof}
Let $f\in H$ be an odd function not approximable by odd polynomials.
Since the partial sums $s_n(f)$ of its Taylor series are odd polynomials, 
their closed linear span does not contain $f$.
Evidently this precludes the possibility of approximating $f$
by polynomials of the form \eqref{E:triangular}.
\end{proof}

The formula \eqref{E:triangular} describes a so-called triangular summability method.
The conclusion of Corollary~\ref{C:summability} is that no such method applied to the
Taylor partial sums $s_n(f)$ will converge to $f$ in $H$. For background on summability methods,
we refer to \cite{Ha92}.

Exactly the same analysis applies to approximation of $f$ by functions of the form
\begin{equation}\label{E:infinite}
T_n(f):=\sum_{k=0}^\infty c_{nk}s_k(f),
\end{equation}
\emph{provided} that these series converge in $H$. 
This is where the last part of Theorem~\ref{T:horrible} comes in.
We have the following result.

\begin{corollary}\label{C:infinite}
Let $H$  and $(\omega_n)$ be as in Theorem~\ref{T:horrible}.
Suppose, in addition, that the sequence $(\omega_n)$ is chosen so that $\lim_{n\to\infty}\omega_n^{1/n}= 1$.
Then, despite the fact that polynomials are dense in $H$,
 there exists $f\in H$ such that:
\begin{itemize} 
\item $\sum_{k=0}^\infty c_{k}s_k(f)$ converges in $H$ whenever $\limsup_{k\to\infty}|c_k|^{1/k}<1$,
\item $f$ lies outside the closed linear span of the set of all such series.
\end{itemize}
\end{corollary}

\begin{proof}
Let $f$ be chosen as in Corollary~\ref{C:summability}.
Since
\[
\limsup_{n\to\infty}\|z^n\|_H^{1/n}\le\limsup_{n\to\infty}(1+\omega_n)^{1/n}\le1, 
\]
elementary estimates give that $\|s_k(f)\|=O(R^k)$ for each $R>1$,
and consequently  $\sum_{k\ge0}|c_k|\|s_k(f)\|_H<\infty$ whenever $\limsup_{k\to\infty}|c_k|^{1/k}<1$.
This establishes the first point, and the second one follows immediately from the choice of $f$ and the fact
every sum $\sum_{k=0}^\infty c_{k}s_k(f)$ belongs to the closed linear span of $\{s_k(f):k\ge0\}$.
\end{proof}

An important special case of the above is the so-called Abel summation method,
where one takes $c_k:=(1-r)r^k$ for $r<1$. Clearly we have $\lim_{k\to\infty}|c_k|^{1/k}=r<1$,
so Corollary~\ref{C:infinite} applies. A simple calculation gives that
\[
\sum_{k\ge0}(1-r)r^ks_k(f)=f_r,
\]
where $f_r(z):=f(rz)$, a radial dilate of $f$. We thus obtain the following corollary.

\begin{corollary}\label{C:dilates}
Let $H$  and $(\omega_n)$ be as in Theorem~\ref{T:horrible},
and suppose that the sequence $(\omega_n)$ is chosen so that $\lim_{n\to\infty}\omega_n^{1/n}= 1$.
Then, despite the fact that polynomials are dense in $H$, 
there exists $f\in H$ lying outside the closed linear span of its set of radial dilates $\{f_r:0<r<1\}$. 
\end{corollary}


\section{Proof of Theorem~\ref{T:horrible}}\label{S:proof}

The proof of Theorem~\ref{T:horrible} is an adaptation of a construction outlined in 
\cite[Proposition~1.34]{HMVZ08}, where it is attributed to W.~B.~Johnson.
We begin with the following lemma.

\begin{lemma}\label{L:Mbasis}
Let $(e_n)_{n\ge0}$ be the standard unit vector basis of $\ell^2=\ell^2(\ZZ^+)$.
Let $(a_n)_{n\ge0}$ and $(b_n)_{n\ge1}$ be two sequences of complex numbers.
For $n\ge0$, define
\begin{align*}
&\begin{cases}
x_{2n}:=e_{2n},\\
x_{2n+1}:=e_{2n+1}-a_ne_{2n}+b_ne_{2n-2},
\end{cases}\\
&\begin{cases}
y_{2n}:=e_{2n}+a_ne_{2n+1}-b_{n+1}e_{2n+3},\\
y_{2n+1}:=e_{2n+1},
\end{cases}
\end{align*}
where, for convenience, we write $e_{-2}:=0$. Then $(x_n,y_n)_{n\ge0}$ is a M-basis for $\ell^2$, i.e.,
\begin{itemize}
\item it is a biorthogonal system: $\langle x_n,y_m\rangle=\delta_{nm}$ for all $n,m\ge0$,
\item each of the sequences $(x_n)_{n\ge0}$ and $(y_n)_{n\ge0}$ spans a dense subspace of $\ell^2$.
\end{itemize}
\end{lemma}

\begin{proof}
First we prove biorthogonality. We need to check four cases:
\begin{align*}
\langle x_{2n},y_{2m}\rangle&=\delta_{nm},
&\langle x_{2n+1},y_{2m+1}\rangle&=\delta_{nm},\\
\langle x_{2n},y_{2m+1}\rangle&=0,
&\langle x_{2n+1},y_{2m}\rangle&=0.
\end{align*}
The first three are obviously true. For the fourth one, we calculate:
\begin{align*}
\langle x_{2n+1},y_{2m}\rangle
&=\bigl\langle e_{2n+1}-a_ne_{2n}+b_ne_{2n-2},\, e_{2m}+a_me_{2m+1}-b_{m+1}e_{2m+3}\bigr\rangle\\
&=a_m\delta_{2n+1,2m+1}-a_n\delta_{2n,2m}-b_{m+1}\delta_{2n+1,2m+3}+b_n\delta_{2n-2,2m}\\
&=0.
\end{align*}

To show that $(x_n)_{n\ge0}$ spans a dense subspace of $\ell^2$,
observe that, for each $n\ge0$, we have $e_{2n}=x_{2n}$ and 
$e_{2n+1}\in \spn\{x_{2n+1}, e_{2n},e_{2n-2}\}=\spn\{x_{2n+1},x_{2n},x_{2n-2}\}$ 
(where $x_{-2}:=0$).
Therefore the span of  $(x_n)_{n\ge0}$ contains the span of $(e_n)_{n\ge0}$,
so it is indeed dense in $\ell^2$.
Likewise, for each $n\ge0$, 
we have $e_{2n+1}=y_{2n+1}$ and  $e_{2n}\in\spn\{y_{2n},e_{2n+1},e_{2n+3}\}=\spn\{y_{2n},y_{2n+1},y_{2n+3}\}$, 
so $(y_n)_{n\ge0}$ spans also a dense subspace of $\ell^2$.
\end{proof}

\begin{proof}[Proof of Theorem~\ref{T:horrible}]
Let $(\eta_n)_{n\ge0}$ be a sequence such that $\eta_n>0$ and $\sum_n\eta_n^2<\infty$,
to be chosen later.
Set $a_n:=1/\eta_n^2~(n\ge0)$ and $b_n:=1/\eta_n\eta_{n-1}~(n\ge1)$, 
and let $(x_n,y_n)_{n\ge0}$ be the M-basis for $\ell^2$ constructed in Lemma~\ref{L:Mbasis}.
For $x\in \ell^2$, define
\[
J(x)(z):=\sum_{n\ge0}\frac{\langle x,y_n\rangle}{\|y_n\|} z^n \quad(z\in\DD).
\]
Since
\[
\sum_{n\ge0}\Bigl|\frac{\langle x,y_n\rangle}{\|y_n\|} z^n\Bigr|
\le \|x\|\sum_{n\ge0}|z|^n
=\frac{\|x\|}{1-|z|} \quad(z\in\DD),
\]
we see that $J(x)\in\hol(\DD)$ for each $x\in\ell^2$, and $J:\ell^2\to\hol(\DD)$ is a continuous linear map.
Moreover, since the sequence $(y_n)$ spans a dense subspace of $\ell^2$, it follows that $J$ is injective.
Define $H:=J(\ell^2)$, with the inner product $\langle\cdot,\cdot\rangle_H$ inherited from $\ell^2$. 
Then  the inclusion $H\hookrightarrow\hol(\DD)$ is continuous,
in other words,  $H$ is a Hilbert holomorphic function space on $\DD$.
Furthermore, since the sequence $(x_n)$ spans a dense subspace of $\ell^2$ 
and $J(\|y_n\|x_n)=z^n$ for each $n$, 
it follows that $H$ contains the polynomials and that polynomials are dense in $H$.

We now show that the odd polynomials are not dense in the odd functions of $H$.
To do this, it suffices to construct $f,g\in H$ such that $f$ is odd, $\langle z^{2n+1},g\rangle_H=0$ for all $n\ge0$,
yet $\langle f,g\rangle_H\ne0$. We do this as follows. Define $u,v\in\ell^2$ by
\[
u:=\sum_{j\ge0} \eta_j e_{2j+1}
\quad\text{and}\quad
v:=\frac{1}{\eta_0}e_1+\sum_{k\ge0} \eta_k e_{2k},
\]
and set $f:=J(u)$ and $g:=J(v)$.
For each $n\ge0$, the coefficient $\widehat{f}(2n)$ of $z^{2n}$ in the Taylor expansion of $f$ satisfies 
\begin{align*}
\|y_{2n}\|\widehat{f}(2n)
&=\langle u,y_{2n}\rangle\\
&=\sum_{j\ge0}\eta_j\langle e_{2j+1},y_{2n}\rangle\\
&=\sum_{j\ge0} \eta_j \langle e_{2j+1},~e_{2n}+a_ne_{2n+1}-b_{n+1}e_{2n+3}\rangle\\
&=\eta_n a_n-\eta_{n+1}b_{n+1}=0.
\end{align*}
Thus $\widehat{f}(2n)=0$ for all $n$, and $f$ is an odd function. Also, for each $n\ge0$, we have
\begin{align*}
\|y_{2n+1}\|^{-1}\langle z^{2n+1}, g\rangle_H
&=\langle x_{2n+1},v\rangle\\
&=\langle x_{2n+1},e_1/\eta_0\rangle
+\sum_{k\ge0}\eta_k\langle x_{2n+1}, e_{2k}\rangle\\
&=\langle e_{2n+1}-a_ne_{2n}+b_ne_{2n-2}, ~e_1/\eta_0\rangle\\
&\quad+ \sum_{k\ge0}\eta_k \langle e_{2n+1}-a_ne_{2n}+b_ne_{2n-2}, ~e_{2k}\rangle\\
&=
\begin{cases}
-\eta_n a_n+\eta_{n-1}b_n, &n\ge1,\\
1/\eta_0-\eta_0a_0, &n=0,
\end{cases}\\
&=0.
\end{align*}
Lastly, we have
\[
\langle f,g\rangle_H=\langle u,v\rangle 
=\Bigl\langle \sum_{j\ge0}\eta_je_{2j+1},~(1/\eta_0)e_1+\sum_{k\ge0}\eta_ke_{2k}\Bigr\rangle
=\frac{\eta_0}{\eta_0}=1\ne0.
\]
This completes the proof that odd polynomials are not dense in the odd functions in $H$.

Finally, we turn to the question of the estimation of $\|z^n\|_H$. Since $z^n=J(\|y_n\|x_n)$, we have
\[
\|z^n\|_H=\Big\|\|y_n\|x_n\Bigr\|=\|x_n\|\|y_n\|.
\]
Recalling the construction of $x_n$ and $y_n$, we deduce the following estimates:
\begin{align*}
\|z^{2n}\|_H&\le 1+|a_n|+|b_{n+1}|=1+\frac{1}{\eta_n^2}+\frac{1}{\eta_n\eta_{n+1}} \quad(n\ge0)\\
\|z^{2n+1}\|_H&\le 1+|a_n|+|b_n|=1+\frac{1}{\eta_n^2}+\frac{1}{\eta_n\eta_{n-1}} \quad(n\ge1)\\
\|z\|_H&\le 1+|a_0|=\frac{1}{\eta_0^2}.
\end{align*}
In order to have $\|z^n\|_H\le 1+ \omega_n$ for all $n$, it therefore suffices that
\[
\frac{3}{\eta_n^2}\le \omega_{2n},
\quad
\frac{1}{\eta_{n+1}^2}\le \omega_{2n},
\quad
\frac{3}{\eta_{n}^2}\le \omega_{2n+1},
\quad
\frac{1}{\eta_{n-1}^2}\le \omega_{2n+1}.
\]
To achieve this, we therefore define $(\eta_n)$ by
\[
\eta_n^2:=
\frac{3}{\omega_{2n}}+\frac{1}{\omega_{2n-2}}+\frac{3}{\omega_{2n+1}}+\frac{1}{\omega_{2n+3}},
\]
with the second term omitted in the case $n=0$. The condition that $\sum_n1/\omega_n<\infty$ then guarantees
that $\sum_n\eta_n^2<\infty$, as required. 
The proof is complete.
\end{proof}

\section{Variations on a theme}\label{S:variations}

Does Theorem~\ref{T:horrible} hold with even functions instead of odd ones? Certainly! 
In fact, the arithmetic structure of the odd/even numbers plays no real role.
As the following result shows, all that matters is that these are infinite subsets of $\ZZ^+$
with infinite complement.

In  what follows, we write $\widehat{f}(n)$ for the coefficient of $z^n$ in the Taylor expansion of $f$.
Also, $\supp\widehat{f}:=\{n\in\ZZ^+:\widehat{f}(n)\ne0\}$.

\begin{theorem}
Let  $I$ be any subset of $\ZZ^+$ such that both $I$ and $\ZZ^+\setminus I$ are infinite.
Then there exists a Hilbert holomorphic function space $H$ on $\DD$ in which the polynomials
are dense, but the set of polynomials $p$ with $\supp\widehat{p}\subset I$
is not dense in  $\{f\in H:\supp \widehat{f}\subset I\}$. 
\end{theorem}

\begin{proof}
Let $\sigma$ be a permutation of $\ZZ^+$ taking the odd integers to $I$ 
and the even ones to $\ZZ^+\setminus I$.
Then, repeating the construction in \S\ref{S:proof} with $J$ defined by
\[
J(x)(z):=\sum_{n\ge0}\frac{\langle x,y_n\rangle}{\|y_n\|} z^{\sigma(n)} \quad(z\in\DD),
\]
we obtain  a Hilbert holomorphic function space $H$ 
with the required properties.
\end{proof}

\begin{remarks}
(i) It is not hard to see that the conditions that $I$ and $\ZZ^+\setminus I$ 
be infinite are both necessary for this theorem to hold.

(ii) The estimate on monomials now becomes $\|z^{\sigma(n)}\|_H\le 1+\omega_n$.
\end{remarks}

By varying the definition of $J$, it is possible to embed $H$ in function spaces 
other than $\hol(\DD)$. There are many possibilities.
The following theorem illustrates the general idea in the context of Fourier series.

We write $\TT$ for the unit circle, $L^2(\TT)$ for the usual space of square-integrable
functions on $\TT$, and $H^2$ for the Hardy space, now considered as a closed subspace of $L^2(\TT)$.

\begin{theorem}\label{T:Fourier}
There exists a Hilbert space $H$ continuously embedded in $L^2(\TT)$ such that
\begin{itemize}
\item $H$ contains the trigonometric polynomials;
\item the trigonometric polynomials are dense in $H$;
\item the holomorphic polynomials are \emph{not} dense in $H\cap H^2$.
\end{itemize}
\end{theorem}

\begin{proof}
Let $\sigma$ be a bijection of $\ZZ^+$ onto $\ZZ$ mapping the odd integers onto $\ZZ^+$
and the even integers onto $\ZZ\setminus\ZZ^+$. 
Then, repeating the construction in \S\ref{S:proof} with $J$ defined by
\[
J(x)(e^{it}):=\sum_{n\ge0}\frac{\langle x,y_n\rangle}{\|y_n\|} \frac{e^{i\sigma(n)t}}{2^n} \quad(e^{it}\in\TT),
\]
we obtain a Hilbert space $H$ continuously embedded in $L^2(\TT)$ 
with the required properties.
\end{proof}

This theorem has the following consequence.
Here $\widehat{f}(k)$ denotes the $k$-th Fourier coefficient of $f$.

\begin{corollary}
Let $H$ be as in Theorem~\ref{T:Fourier}.
Then, despite the fact that trigonometric polynomials are dense in $H$,
there exists $f\in H$ lying outside the closed linear span of
the functions $s_n(f):=\sum_{k=-n}^n \widehat{f}(k)\exp(ikt)$.
\end{corollary}

\begin{proof}
Let $f$ be a function in $H\cap H^2$ not approximable in $H$ by holomorphic polynomials.
Since $f\in H^2$, each of the functions $s_n(f)$ is a holomorphic polynomial,
and so their closed linear span does not contain $f$.
\end{proof}

\bibliographystyle{plain}
\bibliography{biblist}

\begin{thebibliography}{10}

\bibitem{CR00}
J.~A. Cima and W.~T. Ross.
\newblock {\em The backward shift on the {H}ardy space}, volume~79 of {\em
  Mathematical Surveys and Monographs}.
\newblock American Mathematical Society, Providence, RI, 2000.

\bibitem{EFKMR16}
O.~El-Fallah, E.~Fricain, K.~Kellay, J.~Mashreghi, and T.~Ransford.
\newblock Constructive approximation in de {B}ranges-{R}ovnyak spaces.
\newblock {\em Constr. Approx.}, 44(2):269--281, 2016.

\bibitem{EKMR14}
O.~El-Fallah, K.~Kellay, J.~Mashreghi, and T.~Ransford.
\newblock {\em A primer on the {D}irichlet space}, volume 203 of {\em Cambridge
  Tracts in Mathematics}.
\newblock Cambridge University Press, Cambridge, 2014.

\bibitem{HMVZ08}
P.~H\'{a}jek, V.~Montesinos~Santaluc\'{\i}a, J.~Vanderwerff, and V.~Zizler.
\newblock {\em Biorthogonal systems in {B}anach spaces}, volume~26 of {\em CMS
  Books in Mathematics/Ouvrages de Math\'{e}matiques de la SMC}.
\newblock Springer, New York, 2008.

\bibitem{Ha92}
G.~H. Hardy.
\newblock {\em Divergent series}.
\newblock \'{E}ditions Jacques Gabay, Sceaux, 1992.
\newblock With a preface by J. E. Littlewood and a note by L. S. Bosanquet,
  Reprint of the revised (1963) edition.

\bibitem{HKZ00}
H.~Hedenmalm, B.~Korenblum, and K.~Zhu.
\newblock {\em Theory of {B}ergman spaces}, volume 199 of {\em Graduate Texts
  in Mathematics}.
\newblock Springer-Verlag, New York, 2000.

\bibitem{MR19b}
J.~Mashreghi and T.~Ransford.
\newblock Hadamard multipliers on weighted {D}irichlet spaces.
\newblock {\em Integral Equations Operator Theory}, 91(6):Paper No. 52, 13,
  2019.

\bibitem{MR19a}
J.~Mashreghi and T.~Ransford.
\newblock Linear polynomial approximation schemes in {B}anach holomorphic
  function spaces.
\newblock {\em Anal. Math. Phys.}, 9(2):899--905, 2019.

\bibitem{Ni19}
Nikola\"{\i} Nikolski.
\newblock {\em Hardy spaces}, volume 179 of {\em Cambridge Studies in Advanced
  Mathematics}.
\newblock Cambridge University Press, Cambridge, french edition, 2019.

\bibitem{RS91}
S.~Richter and C.~Sundberg.
\newblock A formula for the local {D}irichlet integral.
\newblock {\em Michigan Math. J.}, 38(3):355--379, 1991.

\bibitem{Sa94}
D.~Sarason.
\newblock {\em Sub-{H}ardy {H}ilbert spaces in the unit disk}, volume~10 of
  {\em University of Arkansas Lecture Notes in the Mathematical Sciences}.
\newblock John Wiley \& Sons, Inc., New York, 1994.
\newblock A Wiley-Interscience Publication.

\end{thebibliography}

\end{document}